\documentclass[12pt,leqno]{amsart}
\usepackage{amssymb}
\usepackage{amsmath,amscd}
\usepackage[initials,nobysame,alphabetic]{amsrefs}
\usepackage{amsmath, amssymb}
\usepackage{amsfonts}
\usepackage{mathrsfs}
\usepackage{mathpazo}
\usepackage[arrow,matrix,curve,cmtip,ps]{xy}

\usepackage{amsthm}

\usepackage{float}
\usepackage{hyperref}
\usepackage{graphics}
\usepackage{graphicx}
\usepackage{verbatim}
\usepackage{multirow}
\usepackage{tikz}
\usepackage{enumerate}
\usepackage{eufrak}
\allowdisplaybreaks

\usepackage{kantlipsum}
\calclayout

\newtheorem{theorem}{Theorem}[section]
\newtheorem{lemma}[theorem]{Lemma}
\newtheorem{proposition}[theorem]{Proposition}
\newtheorem{corollary}[theorem]{Corollary}

\newtheorem*{theorem*}{Theorem}
\newtheorem{remark}[theorem]{Remark}
\newtheorem{definition}[theorem]{Definition}

\newtheorem{condition}{Condition}

\def\as{\hbox{\rm -a.s.{ }}}
\def\tr{\hbox{\rm tr$\,$}}

\numberwithin{equation}{section}

\newcommand{\E}{\mathbb{E}}
\newcommand{\R}{\mathbb{R}}
\newcommand{\Ff}{\mathbb{F}}
\newcommand{\U}{\mathcal{U}}

\newcommand{\F}{\mathcal{F}}

\newcommand{\cR}{\mathcal{R}}

\newcommand{\cP}{\mathcal{P}}

\newcommand{\M}{\mathcal{M}}
\newcommand{\hH}{\mathcal{H}}

\def\no{\noindent}

\def\ms{\medskip}

\def\qq{\qquad}

\def\dbR{{\mathop{\rm l\negthinspace R}}}

\def\3n{\negthinspace \negthinspace \negthinspace }
\def\2n{\negthinspace \negthinspace }
\def\1n{\negthinspace }

\def\dbL{{\mathop{\rm l\negthinspace L}}}
\def\dbM{{\mathop{\rm l\negthinspace M}}}

\def\dbP{{\mathop{\rm l\negthinspace P}}}
\def\dbQ{{\mathbb{Q}}}
\def\dbR{{\mathop{\rm l\negthinspace R}}}

\def\={\buildrel \triangle \over =}

\def\bt{\begin{theorem}}
\def\bcd{\begin{condition}}
\def\ecd{\end{condition}}
\def\et{\end{theorem}}
\def\bc{\begin{corollary}}
\def\ec{\end{corollary}}
\def\bde{\begin{definition}}
\def\ede{\end{definition}}
\def\bl{\begin{lemma}}
\def\el{\end{lemma}}
\def\bp{\begin{proposition}}
\def\ep{\end{proposition}}
\def\br{\begin{remark}}
\def\er{\end{remark}}
\def\ba{\begin{array}}
\def\ea{\end{array}}
\def\be{\begin{equation}}
\def\ee{\end{equation}}

\newcommand{\Prob}{\mathbb{\Prob}}

\newcommand{\mytilde}{\raise.17ex\hbox{$\scriptstyle\mathtt{\sim}$}}

\def\d{\delta}             
\def\ep{\varepsilon}       

\def\m{\mu}

\def\ms{\medskip} 
\def\no{\noindent}

\begin{document}
\title[Stochastic Relaxed Optimal Control for G-SDEs]{On Relaxed Stochastic Optimal Control for Stochastic Differential Equations Driven by G-Brownian Motion}

\author{ Amel Redjil and  Salah Eddine Choutri}

\address{Lab. of Probability and Statistics (LaPS), Department of Mathematics \\ Badji Mokhtar university, B.P.12 \\ 23000, Annaba \\ Algeria}  

\email{rdj$\_$amel@yahoo.fr}

\address{Department of Mathematics \\ KTH Royal Institute of Technology \\ 100 44, Stockholm \\ Sweden}

\email{choutri@kth.se}

\date{\today}



\begin{abstract}
In the G-framework, we establish existence of an optimal stochastic relaxed control for stochastic differential equations driven by a G-Brownian motion.
\\

\end{abstract}

\maketitle

\tableofcontents


\section{Introduction}

The purpose of this paper is to study optimal control of systems subject to model uncertainty or ambiguity due to incomplete or inaccurate information, or vague concepts and principles. Climate or weather and financial markets are typical fields where information is subject to uncertainty. For example in optimal portfolio choice problems in finance where the volatility and the risk premium processes are unknown and hard to accurately estimate  from reliable data, we need to consider a family of different models or scenarios instead of one fixed asset process based on a given prior or estimate. To cope with any skeptical attitude towards a given model and quantify ambiguity aversion, the decision maker needs to perform a robust portfolio optimization that survives all given scenarios. 

Aspects of model ambiguity such as volatility uncertainty have been studied by Peng \cite{peng2007, peng2008,peng2010} who introduced an abstract sublinear or $G$-expectation space with a process called $G$-Brownian motion, and by Denis and Martini \cite{denis1} who suggested a structure based on quasi-sure analysis from abstract potential theory to construct a similar structure using a tight family $\cP$ of possibly mutually singular probability measures. Although these two approaches are substantially different, Denis {\it et al.} \cite{denis2} show that they are very closely related by providing a dual representation of the sublinear expectation $\widehat{\E}$ associated with the $G$-Brownian motion as the supremum of ordinary expectations over $\cP$. 

Within the $G$-Brownian motion framework as presented in \cite{peng2007, peng2010}, this paper deals with optimal control of systems governed by stochastic differential equation driven by a $G$-Brownian motion. More precisely, we ask if there exists a stochastic control $\widehat{u}\in\mathcal{U}$ with values in an action set $U$ such that 

\be\label{intro-opt}
J(\widehat{u})=\inf_{u\in \mathcal{U}}J(u),
\ee
with

\be\label{intro-j}\begin{array}{lll}
J(u)=\underset{\dbP\in\cP}\sup J^{\dbP}(u):=\underset{\dbP\in\cP}\sup\E^{\dbP}\left[\int_0^T f(t,x^{u}(t),u(t))dt+h(x^{u}(T))\right]
\end{array}
\ee
where $x^u$ is a $G$-sde given by \eqref{g-sde}, below. This problem has been studied in \cite{hu2013, biagini2014} and \cite{matoussi2015}, where the authors suggest necessary and sufficient optimality conditions in terms of respectively a Pontryagin's type maximum principle and dynamic programming principle. The objective of this work is  to investigate the problem of existence of an optimal control. Knowing that in the absence of convexity assumptions the control problem \eqref{intro-opt} may not have a solution simply because $\mathcal{U}$ is too small to contain a minimizer, we would like to find a set $\cR$ of controls that 'contains' $\mathcal{U}$ and has a richer topological structure for which the control problem becomes solvable. This embedding is often called a relaxation of the control problem and $\cR$ is the set of relaxed controls, while $\mathcal{U}$ is called the set of strict controls. In Theorem \eqref{opt-relax} which constitutes the main result of this paper, we construct the set $\cR$  of relaxed controls as a subset of the set of probability measures on the action set $U$ and show that
\begin{equation*}
\inf_{u \in \mathcal{U}} J(u)=\inf_{\mu \in \cR}J(\mu)=J(\widehat{\mu}),
\end{equation*}
where
\begin{equation*}
\widehat{\mu}=arg\min_{\mu \in \cR}J(\mu).
\end{equation*}

In section 2,  we collect the notions and results from $G$-stochastic calculus needed to establish our result. In section 3, we introduce the space of relaxed controls and its properties. Finally,  in section 4, we consider the relaxed optimal control problem within the G-framework and prove existence of an optimal relaxed control for our system of $G$-sdes.

\section{Preliminaries} In this section we recall the notions and main results from  the framework of G-stochastic calculus, mainly based on the references \cite{denis1, denis2, peng2007, peng2008,peng2010, soner1} and  \cite{soner2},  we will use in this paper. 
\subsection{G-expectation and G-Brownian motion}
Let $\Omega := \{\omega \in C(\mathbb{R}_+,\mathbb{R}^d): \omega(0)=0 \}$, equipped with the topology of uniform convergence on compact intervals, $\mathcal{B}(\Omega)$ the associated Borel $\sigma$-algebra,  $\Omega_t:= \{ w_{.\wedge t}: w \in \Omega \}$, $ B$  the canonical process and $\mathbb{P}_0$ be the Wiener measure on $\Omega$. Let $\mathbb{F}:= \mathbb{F}^B=\{ \mathcal{F}_t\}{_{t \geq 0}}$ be the raw filtration generated by $B$, which is only left-continuous.  Further, consider the right-limit filtration $\mathbb{F^+}:=\{ \mathcal{F}_t^+ , t \geq 0 \},$  where $ \mathcal{F}_t^+:=\mathcal{F}_{t+}:=\cap_{s>t} \mathcal{F}_s$.

Given a probability measure $\dbP$ on $(\Omega,\mathcal{B}(\Omega))$, we consider the right-continuous $\dbP$-completed filtrations $\mathcal{F}_t^{\mathbb{P}}:= \mathcal{F}_t^+ \vee \mathcal{N}^{\mathbb{P}}(\mathcal{F}_t^+)$ and $\widehat{\mathcal{F}}_t^{\mathbb{P}}:=\mathcal{F}_t^+ \vee \mathcal{N}^{\mathbb{P}}(\mathcal{F}_{\infty}),  $ where  the $\dbP$-negligible set $\mathcal{N}^{\mathbb{P}}(\mathcal{G})$ on a $\sigma$-algebra $\mathcal{G}$ is defined as 
$$
\mathcal{N}^{\mathbb{P}}(\mathcal{G}):=\{D \subset \Omega: \text{there \  exists} \ \widetilde{D} \in \mathcal{G} \ \ \text{such \  that}  \ D \subset \widetilde{D} \ \text{and} \ \mathbb{P}[\widetilde{D}]=0  \}. 
$$ 
We have 
\begin{lemma} [Lemma 2.1, \cite{soner1}] Let $\dbP$  be an arbitrary probability measure on $(\Omega,\F_{\infty})$. For every $\widehat{\mathcal{F}}_t^{\mathbb{P}}$-measurable random variable $\widehat{\xi}$, there exists a $\dbP\as$ unique $\F_t$-measurable random variable $\xi$ such that $\xi=\widehat{\xi},\ \dbP\as$ 

For every $\widehat{\Ff}^{\mathbb{P}}$-progressively measurable process $\widehat{X}$, there exists a unique $\Ff$-progressively measurable process $X$ such that $X=\widehat{X}$, $dt\times \dbP\as$ Moreover, if $\widehat{X}$ is $\dbP$-almost surely continuous, then $X$ can be chosen to $\dbP$-almost surely continuous.
\end{lemma}  

 The $G$-expectation is defined by Peng in \cite{peng2007, peng2008, peng2010} through the nonlinear heat equation in the following sense. A $d$-dimensional random vector $X$ is said to be $G$-normally distributed under the $G$-expectation $\widehat{\E}[\cdot]$ if for each  bounded and Lipschitz continuous function $\varphi$ on $\R^d$, $\varphi\in\text{Lip}(\R^d)$, the function $u$ defined by 
$$
u(t,x):=\widehat{\E}[\varphi(x+\sqrt{t}X)],\quad t\ge 0, \,\, x\in\R^d
$$
is the unique, bounded Lipschitz continuous viscosity solution of the following parabolic equation
$$
\frac{\partial u}{\partial t}-G(D^2u)=0 \quad \text{on}\,\, (t,x)\in(0,+\infty)\times \R^d\quad \text{and}\quad u(0,x)=\varphi(x),
$$
where $D^2u=(\partial_{x_ix_j}^2u)_{1\le i,j\le d}$ is the Hessian matrix of $u$ and the nonlinear operator $G$ is defined by
\be\label{G}
G(A):=\frac{1}{2}\sup_{\gamma\in\Gamma}\{ \tr(\gamma\gamma^* A)\},\quad \gamma\in \R^{d\times d}.
\ee
where $A$ is a $d\times d$ symmetric matrix and $\Gamma$ is a given non empty, bounded and closed subset of $\R^{d\times d}$.  Here, $v^*$ denotes the transpose of the vector $v$. This $G$-normal distribution is denoted by $ N(0,\Sigma)$, where $\Sigma:=\{\gamma\gamma^*,\,\, \gamma\in\Gamma\}$.

 Peng \cite{peng2007, peng2008} shows  that the $G$-expectation $\widehat{\E}:\,\mathcal{H}:=\text{Lip}(\R^d)\longrightarrow \R $ is a consistent sublinear expectation on the lattice $\mathcal{H}$ of real functions i.e. it satisfies
\begin{enumerate}
\item Sub-additivity: for all $X,Y \in \mathcal{H},\quad \widehat{\E}[X+Y] \leq \widehat{\E}[X]+\widehat{\E}[Y].$
\item Monotonicity:  for all $ X,Y \in \mathcal{H},\quad  X \geq Y \Rightarrow \widehat{\E}[X] \geq \widehat{\E}[Y].$
\item Constant preserving: for all $c \in \mathbb{R},\quad  \widehat{\E}[c]=c.$
\item Positive homogeneity: for all $\lambda \geq 0,\, X \in \mathcal{H}, \quad \widehat{\E}[\lambda X]=\lambda \widehat{\E}[X].$
\end{enumerate}

 Let $\text{Lip}(\Omega)$ be the set of random variables of the form  $\xi:=\varphi(B_{t_1},B_{t_2},\ldots,B_{t_n})$ for some bounded Lipschitz continuous function $\phi$  on $\R^{d\times n}$ and $0\le t_1\le t_2\le \cdots \le t_n\le T$. The coordinate process $(B_t,\,\,t\ge 0)$ is called $G$-Brownian motion whenever $B_1$ is  $G$-normally distributed under $\widehat{\E}[\cdot]$ and for each $s,t\ge 0$ and $t_1,t_2,\ldots,t_n\in [0,t]$ we have
$$
\widehat{\E}[\varphi(B_{t_1},\ldots,B_{t_n},B_{t+s}-B_t)]=\widehat{\E}[\psi(B_{t_1},\ldots,B_{t_n})],
$$
where $\psi(x_1,\ldots,x_n)=\widehat{\E}[\varphi(x_1,\ldots,x_n,\sqrt{s}B_1)]$. This property implies that the increments of the $G$-Brownian motion are independent and that $B_{t+s}-B_{t}$ and $B_s$ are identically $N(0,s\Sigma)$-distributed. 

A remarkable result of Peng \cite{peng2007, peng2008} is that if $\mathcal{H}$ is a lattice of real functions on $\Omega$ such that $\text{Lip}(\Omega)\subset \mathcal{H}$, then the $G$-expectation $\widehat{\E}:\,\mathcal{H}\longrightarrow \R $ is a consistent sublinear expectation. 

For $p\in[0,+\infty)$, denote by $\dbL_G^p(\Omega)$  the closure of $\text{Lip}(\Omega)$ under the Banach norm 
 $$
 \|X\|^p_{\dbL_G^p(\Omega)}:=\widehat{\E}[|X|^p].
 $$ 
For each $t\ge 0$,  let $L^0(\Omega_t)$ be the set of $\F_t$-measurable functions. We set  
$$
\text{Lip}(\Omega_t):= \text{Lip}(\Omega) \cap L^0(\Omega_t),\quad  \dbL_G^p (\Omega_t):=\dbL_G^p(\Omega) \cap L^0(\Omega_t).
$$

\subsection{G-stochastic integrals}
For $p\in[0,+\infty)$, we let $M_G^{0,p} (0,T)$ be the space of $\Ff$-progressively measurable, $\R^d$-valued elementary processes of the form
\begin{equation*}
\eta (t)= \sum_{i=0}^{n-1}\eta_i \mathbb{1}_{[t_i,t_{i+1})}(s),
\end{equation*}
where $0 \leq 0=t_0 <t_1 <\cdots<t_{n-1}<t_n= T,\,\, n \geq 1 $ and $\eta_i \in \text{Lip}(\Omega_{t_i})$.  Let $M_G^p(0,T)$ be the closure of $M_G^{0,p} (0,T)$  under the norm 
\begin{equation*}
\Vert \eta \Vert^p_{M_G^p(0,T)}:= \hat{\mathbb{E}}[\int_0^T \vert \eta(t) \vert^p ds].
\end{equation*}
For each $\eta \in M_G^{0,2}(0,T)$, the $G$-stochastic integral is defined pointwisely by 
\begin{equation*}
I_t(\eta)=\int_0^t \eta_s d_GB_s := \sum_{j=0}^{N-1} \eta_j (B_{t\wedge t_{j+1}}-B_{t\wedge t_j}).
\end{equation*}
With $I(\eta):=I_T(\eta)$, the mapping $I: \,M_G^{0,2}(0,T) \rightarrow \dbL_G^{2}(\Omega_T)$ is continuous and thus can be continuously extended to $M_G^2(0,T).$

The quadratic variation process of G-Brownian motion can be formulated in $\dbL_G^2(\Omega_t) $ by the continuous $d\times d$-symmetric-matrix-valued process defined by
\begin{equation}\label{G-quadra}
\langle B \rangle^G_t := B_t\otimes B_t-2 \int_0^t B_s\otimes d_GB_s, 
\end{equation}
whose diagonal is constituted of nondecreasing processes. Here, for $a,b\in\R^d$, the $d\times d$-symmetric matrix $a\otimes b$ is defined by $(a\otimes b)x=(a\cdot x)b$ for $x\in\R^d$, where $"\cdot"$ denotes the scalar product in  $\R^d$.

Define the mapping $\mathcal{J}:M_G^{0,1}(0,T) \mapsto \mathbb{L}_G^1(\Omega_T)$:
$$
\mathcal{J}=\int_0^T \eta_t d\langle B \rangle^G_t := \sum_{j=0}^{N-1} \eta_j (\langle B \rangle^G_{t_{j+1}}-\langle B\rangle^G_{t_j}).
$$
Then $\mathcal{J}$ can be uniquely extended to $\mathcal{Q}: M_G^1(0,T) \rightarrow \mathbb{L}_G^1(\Omega_T)$, where 
$$
\mathcal{Q}:=\int_0^T \eta_t d\langle B \rangle^G_t,\ \ \eta \in M_G^{1}(0,T).
$$

We have the following 'isometry' (formulated for the case $d=1$, for simplicity).
\begin{lemma} $($\cite{peng2007}$)$ \\
Assume $d=1$ and let $\eta \in M_G^2 (0,T)$. We have
$$  
\widehat{\E} \left[ \left( \int_0^T  \eta(s) d_GB_s  \right)^2   \right] = \widehat{\E} \left[ \int_0^T \eta^2(s) d \langle B \rangle^G_s \right].
$$
\end{lemma}

\subsection{A dual representation of $G$-expectation}
Denis and Martini \cite{denis1} and Denis {\it et al.} \cite{denis2} prove the following dual representation of the $G$-expectation in terms of a weakly compact (tight) family $\mathcal{P}$ of possibly mutually singular probability measures on $(\Omega,\mathcal{B}(\Omega))$. This duality expresses the $G$-expectation as a robust expectation with respect to $\mathcal{P}$.
We refer to \cite{denis1} and \cite{denis2} for explicit constructions of  $\mathcal{P}$. Soner {\it et al.} \cite{soner 1, soner 2} perform an in-depth analysis of such a construction and its consequences on the $G$-stochastic analysis and  in particular the question of aggregation of processes. 

\begin{proposition}\label{denis-1}  $($\cite{denis1, denis2}$)$ 
There exists a family of weakly compact probability measures $\mathcal{P}$ on $(\Omega,\mathcal{B}(\Omega))$ such that for each $\xi \in {\mathbb L}_G^1(\Omega)$
\be\label{duality}
\widehat{\E}[\xi]=\sup_{\dbP \in \cP} \E^{\dbP}[\xi].
\ee
Moreover, the set function
$$
c(A):=\sup_{\dbP \in \cP} \dbP(A),\quad A\in\mathcal{B}(\Omega),
$$
defines a regular Choquet capacity.
\end{proposition}

This leads to the following (cf. \cite{denis1, soner1})
\begin{definition}\label{polar} A set $A\in \mathcal{B}(\Omega)$ is called polar if $c(A)=0$ or equivalently if
$\, \dbP(A)=0$ for all $\dbP\in\cP$.  We say that a property holds $\cP$- quasi-surely (q.s.) if it holds $\dbP$-almost-surely for all $\dbP\in\cP$ i.e. outside a polar set. A probability measure $\mathbb{P} $ is called absolutely continuous with respect to $\mathcal{P} $ if $\mathbb{P}(A)=0$ for all $A\in \mathcal{N}_{\mathcal{P}}.$
\end{definition}

Denote by $\mathcal{N}_{\mathcal{P}}:=\bigcap_{\mathbb{P} \in \mathcal{P}} \mathcal{N}^{\mathbb{P}}(\mathcal{F}_{\infty})$  the $\mathcal{P}$-polar sets. We shall use the following universal filtration $\mathbb{F}^{\mathcal{P}}$ for the possibly mutually singular probability measures $\{ \mathbb{P}, \mathbb{P} \in \mathcal{P} \}$ (cf. \cite{soner2}).
\begin{equation}\label{univ-filt}
\mathbb{F}^{\mathcal{P}}:= \{  \widehat{\mathcal{F}}^{\mathcal{P}}_t   \}_{t \geq 0} \quad\quad\text{where} \quad \widehat{\mathcal{F}}_t^{\mathcal{P}}:= \bigcap_{\mathbb{P} \in \mathcal{P}}(\mathcal{F}_t^{\mathbb{P}} \vee \mathcal{N}_{\mathcal{P}}) \quad \text{for} \quad t \geq 0.
\end{equation}

The dual formulation of the $G$-expectation suggests the following aspect of aggregation.
\begin{lemma}[Proposition 3.3, \cite{soner1}]\label{aggregation1}
Let $\eta\in M_G^2(0,T)$. Then, $\eta$ is It\^o-integrable for every $\dbP\in\cP$. Moreover, for every $t\in[0,T]$,
\begin{equation}\label{G-ito}
\int_0^t  \eta (s) d_GB_s=\int_0^t \eta (s) dB_s,\quad \dbP\as \quad\text{for every}\,\,\,\, \dbP\in\cP.
\end{equation}
where the right hand side is the usual It\^o integral. Consequently, the quadratic variation process $\langle B\rangle^G$ defined in \eqref{G-quadra} agrees with the usual quadratic variation process quasi-surely.
\end{lemma}
\no In the sequel, we will drop the notation $G$ from both the $G$-stochastic integral and the $G$-quadratic variation.

Among the stability results derived in \cite{denis2}, we quote the following which plays an important role in our analysis.  
\begin{lemma}[Lemma 29, \cite{denis2}] \label{denis-2} 
If $\{\dbP_n\}_{n=1}^{\infty}\subset\cP$ converges weakly to $\dbP\in\cP$. Then for each $\xi \in {\mathbb L}_G^1(\Omega_T)$, 
${\E}^{\dbP_n}[\xi]\to {\E}^{\dbP}[\xi]$.
\end{lemma}

Considering the properties of the quadratic variation process $\langle B \rangle $ in the G-framework  and the dual formulation of the G-expectation, we have the following Burkholder-Davis-Gundy-type estimates, formulated in one dimension for simplicity.

\begin{lemma} $($\cite{gao2009}$)$ \label{G-BDG}
Assume $d=1$. For each $p \geq 2 $ and $ \eta \in M_G^p(0,T)$, then there exists some constant $C_p$ depending only on $p$ and $T$ such that
$$
\widehat{\E} \left[ \sup_{s \leq u \leq t } \left| \int_s^u \eta_r dB_r \right|^p  \right] \leq C_p  \widehat{\E}\left[\left(\int_s^t \vert \eta_r \vert^2 dr \right)^{p/2}\right]\le C_p|t-s|^{\frac{p}{2}-1}\int_s^t\widehat{\E}[\vert \eta_r \vert^p|] dr.
$$

If $\bar{\sigma}$ is  a positive constant such that  $\frac{d\langle B \rangle_t}{dt}\le \bar{\sigma}$ quasi-surely. Then, for each $p \geq 1 $ and $ \eta \in M_G^p(0,T),$
$$
\widehat{\E} \left[\sup_{ s \leq u \leq t } \left|\int_s^u \eta_r d \langle B \rangle_r \right|^p  \right] \leq \bar{\sigma}^p|t-s|^{p-1} \int_s^t \hat{\mathbb{E}}[\vert \eta_r  \vert^p]dr.
$$
\end{lemma}

\section{The space of relaxed controls}
Let $(U,d)$ be a separable metric space and $\cP(U)$ be the space of probability measures on the set $U$ endowed with its Borel $\sigma$-algebra $\mathcal{B}(U)$. The class $\M([0,T]\times U)$ of relaxed controls  we consider in this paper is a subset of the set $\dbM([0,T]\times U)$ of Radon measures $\nu(dt,da)$ on $[0,T]\times U$ equipped with the topology of stable convergence of measures, whose projections on
$[0,T]$ coincide with the Lebesgue measure $dt$, and whose projection on $U$ coincide with some probability measure $\mu_t(da)\in\cP(U)$ i.e. $\nu(da,dt):=\mu_t(da)dt.$
The topology of stable convergence of measures is the coarsest topology which makes the mapping 
$$
q\mapsto \int_0^T\int_U \varphi(t,a)q(dt,da)
$$ 
continuous, for all bounded measurable functions $\varphi(t,a)$ such that for fixed $t$, $\varphi(t,\cdot)$ is continuous. 

Equipped with this topology, $\dbM:=\dbM([0,T]\times U)$ is a separable metrizable space. Moreover, it is compact whenever $U$ is compact. The topology of stable convergence of measures implies the topology of weak convergence of measures. For further details see \cite{elkaroui1987} and \cite{elkaroui1988}.

Next, we introduce the class of relaxed stochastic controls on $(\Omega_T,\mathcal{H}, \widehat{\E})$, where
$\mathcal{H}$ is a vector lattice of real functions on $\Omega$ such that $\text{Lip}(\Omega_T)\subset \mathcal{H}$.

\begin{definition}[{\bf Relaxed stochastic control}] 
A relaxed stochastic control (or simply a relaxed control) on 
$(\Omega_T,\text{Lip}(\Omega_T),\widehat{\E})$ is  a random measure $q(\omega, dt,da)=\mu_t(\omega,da)dt$ such that for each subset $A\in \mathcal{B}(U)$, the process $(\mu_t(A))_{t\in[0,T]}$ is $\mathbb{F}^{\cP}$-progressively measurable i.e. for every $t\in[0,T]$, the mapping $[0,t]\times\Omega\rightarrow [0,1]$ defined by $(s,\omega)\mapsto \mu_s(\omega,A)$ is $\mathcal{B}([0,t])\otimes \widehat{\mathcal{F}}^{\cP}_t$-measurable. In particular, the process $(\mu_t(A))_{t\in[0,T]}$ is adapted to the {\it universal } filtration $\mathbb{F}^{\cP}$ given by \eqref{univ-filt}. We denote by $\cR$ the class of relaxed stochastic controls.
\end{definition}

The set $\U([0,T])$  of 'strict' controls constituted of  $\mathbb{F}^{\cP}$-adapted processes $u$ taking values in the set $U$, embeds into the set $\cR$ of relaxed controls through the mapping 
\be\label{embed}
\Phi:\quad  \U([0,T])\ni u\mapsto \Phi(u)(dt,da)=\delta_{u(t)}(da)dt\in\cR. 
\ee

The next lemma which extends the celebrated Chattering Lemma, states that each relaxed control in $\cR$ can be approximated with a sequence of strict controls from  $\U([0,T])$. 

\begin{lemma}[{\bf G-Chattering Lemma}] \label{G-chat}
Let $(U,d)$ be a separable metric space and assume that $U$ is a compact set. Let $(\mu_t)_t$ be an $\mathbb{F}^{\cP}$ -progressively measurable process with values in $\cP(U)$. Then there exists a sequence $(u_n(t))_{n\ge 0}$ of $\mathbb{F}^{\cP}$-progressively measurable processes with values in $U$ such that the sequence of random  measures $\d_{u_n(t)}(da)dt$ converges in the sense of stable convergence (thus, weakly) to  $\m_t(da)dt\,\,$ quasi-surely. 
\end{lemma}
\begin{proof} Given the $\mathbb{F}^{\cP}$-progressively measurable relaxed control $\mu$, the detailed pathwise construction  of the approximating sequence $(\delta_{u_n(t)}(da)dt)_{n\ge 0}$ of $\mu_t(da)dt$ in \cite{fleming} (Lemma after Theorem 3) or \cite{elkaroui1988} (Theorem 2.2) extends easily to  make the strict controls $(u_n)_n$  $\,\mathbb{F}^{\cP}$-progressively measurable. \end{proof}

\section{G-Relaxed stochastic optimal control}
In this section we establish existence of a minimizer of the relaxed performance functional
\begin{equation}\label{J}\begin{array}{lll}
J(\mu)=\widehat\E\left[\int_0^T\int_U f(t,x^{\mu}(t),a)\mu_t(da)\,dt+h(x^{\mu}(T))\right]
\end{array}
\end{equation}
over the set $\cR$ of relaxed controls for the relaxed G-sde (written in vector form)
\begin{align}\label{g-sde-r} 
\left\{\begin{array}{lll} 
dx^{\mu}(t)=\sigma(t,x^{\mu}(t))dB_t+\int_U b(t,x^{\mu}(t),a)\m_t(da)dt +\int_U\gamma(t,x^{\mu}(t),a)\m_t(da)d\langle B\rangle_t, \\ x^{\mu}(0)=x.  
\end{array}
\right.
\end{align}
where 
$$
b: [0,T] \times\mathbb{R}^d   \times  U \rightarrow \mathbb{R}^d,\,\,  \sigma, \gamma :[0,T] \times \mathbb{R}^d \times U \rightarrow \mathbb{R}^{d \times d},\,\, 
 f: [0,T] \times \mathbb{R}^d \times  U \rightarrow \mathbb{R},\,\,  h: \R^{d}\rightarrow \R 
$$ 
are deterministic functions.

When  $\mu=\delta_u, u\in\U([0,T])$, the process $x^{\delta_u}:=x^u$ simply solves the following G-sde 
\begin{align}\label{g-sde} 
\left\{\begin{array}{lll} 
dx^{u}(t)=\sigma(t,x^{u}(t))dB_t+b(t,x^{u}(t),u(t))dt +\gamma(t,x^{u}(t),u(t))d\langle B\rangle_t, \\ x^{u}(0)=x.  
\end{array}
\right.
\end{align}
Furthermore, in view of the embedding \eqref{embed}, we may write $J(u)=J(\delta_u)$.

We make the following 

\ms

{\bf Assumptions}
\begin{itemize}
\item[(H1)] The functions $b,\gamma$ and $\sigma $ are continuous and bounded. Moreover, they are Lipschitz continuous with respect to the space variable uniformly in $(t,u)$.

\item[(H2)] the functions $f$ and $h$ are continuous and bounded.
\end{itemize}

\ms
From \cite{peng2010} it follows that under assumption (H1), for each $\mu\in\cR$, the G-sde \eqref{g-sde-r} admits a unique solution $x^{\mu}\in M_G^2(0,T)$ which  satisfies
\begin{equation}
\widehat\E[\sup_{0\le t\le T}|x^{\mu}(t)|^2]<\infty.
\end{equation}
Moreover, for each $\mu\in\cR$,
\begin{equation}\label{rate-J}
\chi^{\mu}:=\int_0^T\int_U f(t,x^{\mu}(t),a)\mu_t(da)\,dt+h(x^{\mu}(T))\in \mathbb{L}^1_{G}(\Omega_T).
\end{equation}
\begin{remark}\label{H1-H2}
Assumptions (H1) and (H2) are strong and can be made much weaker. We impose them to keep the presentation of the main results simple.    
\end{remark}

In this section we prove the following theorem which constitutes the main result of the paper.
\bt\label{opt-relax}
We have
\be\label{inf}
\inf_{u\in \U[0,T]}J(u)\ \ =\inf_{\mu \in \cR}J(\mu ).
\ee
Moreover, there exists a relaxed control  $\widehat{\mu}\in \cR$ such that
\be\label{min}
 J(\widehat{\mu})=\min_{\mu \in \cR}J(\mu).
\ee
\et

Recall that 
\be\label{J-J-P}
J(\mu)=\sup_{\dbP \in \cP}J^{\dbP}(\m),
\ee
where the relaxed performance functional associated to each $\dbP\in \cP$ is given by
\begin{equation}\label{P-J_r}\begin{array}{lll}
J^{\dbP}(\mu)=\E^{\dbP}\left[\int_0^T\int_U f(t,x^{\mu}(t),a)\mu_t(da)\,dt+h(x^{\mu}(T))\right].
\end{array}
\end{equation}

To prove \eqref{inf}, we use the G-Chattering Lemma \eqref{G-chat} and stability results for the G-sde \eqref{g-sde-r}. In view of the G-Chattering Lemma, given a relaxed control $\mu\in\cR$, there exists a sequence $(u^n)_n\in \U([0,T])$ of strict controls such that $\delta_{u^n(t)}(da)dt$ converges weakly to $\mu_t(da)dt$ quasi-surely i.e. $\dbP\as$, for all $\dbP\in\cP$. The proof of \eqref{min} is based of existence of an optimal relaxed control for each $\dbP\in\cP$ and a tightness argument.

Denote by $x^{n}$ the solution of the G-sde associated with $u^n$ (or $\delta_{u^n(t)}(da))$ defined by
\be\label{g-x-n} 
\left\{\begin{array}{lll} 
dx^n(t)=\sigma(t,x^n(t))dB_t+b(t,x^n(t),u^n(t))dt+\gamma(t,x^n(t),u^n(t))d\langle B\rangle_t,  \\ x^n(0)=x. 
\end{array}
\right.
\ee
We have the following stability results for both the G-sde \eqref{g-sde-r} and the performance $J^{\dbP}$, for every $\dbP\in\cP$. 

\begin{lemma} \label{p-x-J-n} For every $\dbP\in\cP$, it holds that 
\be\label{p-x-n}
\lim_{n \to \infty }\E^{\dbP}\left[\sup_{0\leq t\leq T}\left\vert
x^n(t)-x^{\mu}(t)\right\vert ^{2}\right]=0
\ee
and
\be\label{p-J-n}
\lim_{n \to \infty }J^{\dbP}(u^n)=J^{\dbP}(\mu).
\ee
Moreover,
\be\label{p-u-mu}
\inf_{u\in\U[0,T]}J^{\dbP}(u)=\inf_{\mu\in\cR}J^{\dbP}(\mu),
\ee
and there exists a relaxed control $\hat\mu_{\dbP}\in\cR$ such that 
\be\label{p-min}
J^{\dbP}(\hat\mu_{\dbP})=\inf_{\mu\in\cR}J^{\dbP}(\mu).
\ee
\end{lemma}
Since, due to the aggregation property of Lemma \eqref{aggregation1}, under every $\dbP\in \cP$, the G-sde \eqref{g-sde-r} becomes a standard sde, the proof of this result follows from  \cite{bahlali1} or \cite{bahlali2}.

\ms
The purpose of the next proposition is to make the limit \eqref{p-x-n} valid under the sublinear expectation $\widehat{\E}[\cdot]$.

\begin{proposition}\label{g-stability-x} Suppose that $b, \gamma$ and $\sigma$ satisfy condition (H1). Let $x^{\mu}$ and $x^n$ be
the solutions of \eqref{g-sde-r} and \eqref{g-x-n}, respectively. Then
\begin{equation}\label{G-x-n}
\lim_{n\rightarrow \infty }\hat\E\left[ \sup_{0\leq t\leq T}\left\vert
x^n(t)-x^{\mu}(t)\right\vert ^{2}\right] =0.
\end{equation}
\end{proposition}
\begin{proof}
Set $\xi_n:=\sup_{0\leq t\leq T}\left\vert x^n(t)-x^{\mu}(t)\right\vert ^{2}$ and note that for each $n\ge1$, $\xi_n \in\mathbb{L}^1_G(\Omega_T)$. If there is a $\delta>0$ such that $\hat\E[\xi_n]\ge \delta,\, n=1,2,\ldots,$ we then can find a probability $\dbP_n\in\cP$ such that $\E^{\dbP_n}[\xi_n]\ge \delta-\frac{1}{n},\,\, n=1,2,\ldots.$ Since $\cP$ is weakly compact, there exists a subsequence $\{\dbP_{n_k}\}_{k=1}^{\infty}$  that converges weakly to some $\dbP\in\cP$.  We then have
$$
\lim_{j\to\infty}\E^{\dbP}[\xi_{n_j}]=\lim_{j\to\infty}\lim_{k\to\infty}\E^{\dbP_{n_k}}[\xi_{n_j}]\ge \liminf_{k\to\infty}\E^{\dbP_{n_k}}[\xi_{n_k}]\ge \delta.
$$
This contradicts the fact that $\underset{j\to\infty}\lim\E^{\dbP}[\xi_{n_j}]=0$ from Lemma \eqref{p-x-J-n}.
\end{proof}

\begin{remark}\label{dominated convergence} The proof of \eqref{p-x-n} relies on the standard Gronwall and Burkholder-Davis-Gundy inequalities together with the Dominated Convergence Theorem, due to the stable convergence of $\delta_{u_n(t)}(da)dt$ to $\mu_t(da)dt$. Unfortunately, the method of the proof does not extend to prove \eqref{G-x-n}, because under sublinear expectation, the Dominated Convergence Theorem (and even the celebrated Fatou's Lemma) is no longer valid, although the Gronwall and  Burkholder-Davis-Gundy inequalities (see Lemma \eqref{G-BDG}) are still valid for G-sdes and G-Brownian stochastic integrals. 
\end{remark}

\bc \label{g-stability-J} Suppose that $f$ and $h$ satisfy assumption (H2). Let $J(u^{n})$ and $J(\mu )$
be the performance functionals corresponding respectively to $u^{n}$ and $\mu $ where $dt\delta_{u^{n}(t)}(da)$ converges weakly to $dt\mu_t(da)$ quasi-surely. Then, there exists a subsequence $\left(
u^{n_{k}}\right)$ of $\left( u^{n}\right) $ such that 
$$
\lim_{k\to \infty }J(u^{n_k})=J(\mu).
$$
\ec
\begin{proof} From Proposition 17 in \cite{denis2} and Proposition \eqref{g-stability-x} it follows that there exists a subsequence $\left(x^{n_k}(t)\right)_{n_k}$ that converges to $x^{\mu}(t)$ quasi-surely i.e. $\dbP\as$, for all $\dbP\in\cP$, uniformly in $t$. We may apply Lemma \eqref{p-x-J-n} to obtain, for every $\dbP\in\cP$,

\be\label{lim}
\lim_{k\to \infty }J^{\dbP}(u^{n_k})=J^{\dbP}(\mu).
\ee

Using the notation \eqref{rate-J}, we note that $J(u^{n_k})=\widehat\E[\chi^{u^{n_k}}]$ and $J(\mu)=\widehat\E[\chi^{\mu}]$, where both $\chi^{u^{n_k}}$ and $\chi^{\mu}$ belong to $\mathbb{L}^1_G(\Omega_T)$.
If there is some $\delta>0$ such that $\widehat\E[\chi^{u^{n_k}}]\ge \widehat\E[\chi^{\mu}]+\delta,\,n_k\ge \ell,\ell+1,\ldots$, we can then find a probability measure $\dbP_m\in\cP$ such that 
$$
\E^{\dbP_{m}}[\chi^{u^{n_k}}]\ge \widehat\E[\chi^{\mu}]+\delta-\frac{1}{m}.
$$
Since $\cP$ is weakly compact, then we can find a subsequence $\{\dbP_{m_k}\}_{k=1}^{\infty}$ that converges to $\dbP\in\cP$. We then have
\begin{equation*}
\begin{array}{lll}
\E^{\dbP}[\chi^{\mu}]=\underset{k\to\infty}\lim\E^{\dbP_{m_k}}[\chi^{\mu}]=\underset{k\to\infty}\lim\underset{j\to\infty}\lim\E^{\dbP_{m_k}}[\chi^{u_{n_j}}]\ge \underset{j\to\infty}\liminf\, \E^{\dbP_{m_j}}[\chi^{u_{n_j}}]\\\qquad\qquad\ge\underset{j\to\infty}\liminf\left(\widehat\E[\chi^{\mu}]+\delta-\frac{1}{m_j}\right)=\widehat\E[\chi^{\mu}]+\delta.
\end{array}
\end{equation*}
Thus, $\E^{\dbP}[\chi^{\mu}]\ge \widehat\E[\chi^{\mu}]+\delta$, which contradicts the definition of the sublinear expectation. Therefore, 
$$
\lim_{k \to \infty }J(u^{n_k})\le J(\mu).
$$

Next, we prove that $\underset{k\to\infty}\lim J(u^{n_k})\ge J(\mu)$. We have 
$$\begin{array}{lll}
\underset{k\to\infty}\lim J(u^{n_k})\ge \underset{k\to\infty}\lim J^{\dbP}(u^{n_k}), \quad \text{for all } \,\, \dbP\in\cP, \\ \qquad\qquad\quad =J^{\dbP}(\mu), \quad\text{(by \eqref{lim})}, \quad \text{for all } \,\, \dbP\in\cP.
\end{array}
$$
Therefore, $\underset{k\to\infty}\lim J(u^{n_k})\ge J(\mu)$.

\end{proof}

{\bf Proof of Theorem \eqref{opt-relax}}. By \eqref{embed} and Corollary \eqref{g-stability-J} we readily get that
$$
\underset{u\in\U[0,T]}{\inf}J(u)\le \underset{\mu\in\cR}{\inf}J(\mu).
$$
On the other hand since, for every $u\in\U[0,T]$, $\delta_u\in\cR$. Therefore,
$$
J(u)=J(\delta_u)\ge \underset{\mu\in\cR}{\inf}J(\mu),
$$
Hence,
$$
\underset{u\in\U[0,T]}{\inf}J(u)\ge \underset{\mu\in\cR}{\inf}J(\mu).
$$
This proves \eqref{inf}.

\ms
We now turn to the proof of existence of relaxed optimal control. Since $f$ and $h$ are continuous and bounded, for each $\nu\in\cR$
$$
\chi^{\nu}:=\int_0^T\int_U f(t,x^{\nu}(t),a)\nu_t(da)\,dt+h(x^{\nu}(T))\in \mathbb{L}^1_G(\Omega_T).
$$
By Lemma \eqref{denis-2} we then obtain that for every $\nu\in\cR$,
\begin{equation}\label{Q-conv}
\lim_{n\to\infty}J^{\dbP_n}(\nu)=J^{\dbQ}(\nu),
\end{equation}
whenever, the sequence $\{\dbP_n\}_{n=1}^{\infty}\in\cP$ converges weakly to $\dbQ\in\cP$.

Assume that there is an $\varepsilon>0$ such that, for every $\nu\in\cR$, 
$$
J(\nu)\ge\inf_{\mu \in \cR}J(\mu)+\varepsilon.
$$
Since, by Lemma \eqref{p-x-J-n}, for every $\dbP\in\cP$, there exists a relaxed control $\hat\mu\in\cR$ such that 
$\hat\mu_{\dbP}=\arg\min_{\mu\in\cR} J^{\dbP}(\mu)$, we obtain
$$
J(\nu)\ge \sup_{\dbP\in\cP}\inf_{\mu \in \cR}J^{\dbP}(\mu)+\varepsilon= \sup_{\dbP\in\cP}J^{\dbP}(\hat\mu_{\dbP})+\varepsilon.
$$
On the other hand, for every $n\ge 1$, there exists $\dbP_n\in\cP$ such that
$$
J^{\dbP_n}(\nu)\ge J(\nu)+\frac{1}{n}.
$$
The sequence $\{\dbP_n\}_{n=1}^{\infty}\in\cP$ being weakly compact, we can extract a subsequence  $\{\dbP_{n_j}\}_{j=1}^{\infty}\in\cP$ which converges weakly to some $\mathbb{Q}\in\cP$. Thus, it follows from \eqref{Q-conv} that, for every $\nu\in\cR$,
$$
J^{\dbQ}(\nu)=\lim_{j\to\infty} J^{\dbP_{n_j}}(\nu)\ge\sup_{\dbP\in\cP}J^{\dbP}(\hat\mu_{\dbP})+\varepsilon.
$$
In particular, for a given $\nu^{\dbQ}\in\cR$, we obtain
$$
J^{\dbQ}(\nu^{\dbQ})\ge J^{\dbQ}(\nu^{\dbQ})+\varepsilon,
$$
which contradicts the fact that $\varepsilon>0$. \qed

\end{document}